\documentclass{amsart}%
\usepackage{cite}
\usepackage[hypertex]{hyperref}
\usepackage{amssymb}
\usepackage{amsmath}
\usepackage{amsthm}
\usepackage{amsfonts}
\usepackage{graphicx}
\usepackage{bbm}
\usepackage{color}
\usepackage[toc,page]{appendix}%
\providecommand{\U}[1]{\protect\rule{.1in}{.1in}}
\@ifundefined{pdfoutput} { \DeclareGraphicsRule{.wmf}{bmp}{}{}
\DeclareGraphicsRule{.jpg}{bmp}{}{}
\DeclareGraphicsRule{.png}{bmp}{}{}
\DeclareGraphicsRule{.cdr}{bmp}{}{}
\DeclareGraphicsRule{.gif}{bmp}{}{} } { \ifnum\pdfoutput=0\relax
\DeclareGraphicsRule{.wmf}{bmp}{}{}
\DeclareGraphicsRule{.jpg}{bmp}{}{}
\DeclareGraphicsRule{.png}{bmp}{}{}
\DeclareGraphicsRule{.cdr}{bmp}{}{}
\DeclareGraphicsRule{.gif}{bmp}{}{} \fi \ifnum\pdfoutput=1\relax
\DeclareGraphicsRule{.eps}{pdf}{}{}
\DeclareGraphicsRule{.wmf}{jpg}{}{} \fi } \makeatother
\newtheorem{thm}{Theorem}[section]
\newtheorem{cor}[thm]{Corollary}
\newtheorem{ex}[thm]{Example}

\newtheorem{nota}[thm]{Notation}
\newtheorem{prop}[thm]{Proposition}

\theoremstyle{definition}
\newtheorem{df}[thm]{Definition}

\newtheorem{rem}[thm]{Remark}

\numberwithin{equation}{section}
\begin{document}

\title[Brownian and energy representations of path groups]{Equivalence of the Brownian and energy representations}

\author[Albeverio]{Sergio Albeverio{$^{\dag}$}}
\thanks{\footnotemark {$^{\dag }$} Research was supported in part by CIB, EPFL and HCM, University of Bonn.}
\address{$^{\dag }$
Institut f\"{u}r Angewandte Mathematik \\
Abteilung Wahrscheinlichkeitstheorie und HCM \\
Rheinische Friedrich-Wilhelms-Universit\"{a}t Bonn \\
Endenicher Allee 60 \\
53115 Bonn, Germany
}
\email{albeverio@iam.uni-bonn.de}

\author[Driver]{Bruce K. Driver{$^{\ast }$}}
\thanks{\footnotemark {$^\ast$}This research was supported in part by NSF Grant DMS-0739164.}
\address{$^\ast$ Department of Mathematics, 0112\\
University of California, San Diego \\
La Jolla, CA 92093-0112 }
\email{bdriver{@}ucsd.edu}

\author[Gordina]{Maria Gordina{$^{\dag \dag}$}}
\thanks{\footnotemark {$\dag \dag$} Research was supported in part by NSF Grant DMS-1007496.}
\address{$^{\dag \dag}$ Department of Mathematics\\
University of Connecticut\\
Storrs, CT 06269,  U.S.A.}
\email{maria.gordina@math.uconn.edu}

\author[Vershik]{A. M. Vershik{$^{\ddag }$}}
\thanks{\footnotemark {$\ddag $} Research was supported in part by RFFR Grant 14-01-00373.}
\address{$^{\ddag }$
St.Petersburg Department of Steklov Institute of Mathematics \\
Russian Academy of Sciences, \\
Mathematics and Mechanics Department \\
St.Petersburg State University \\
Institue of the Probelm of Transmission of Information \\
Russia }
\email{vershik@pdmi.ras.ru}



%\date{\today \ \emph{File:\jobname{.tex}}}

\begin{abstract}
We consider two unitary representations of the infinite-dimensional groups of smooth paths with values in a compact Lie group. The first representation is induced by quasi-invariance of the Wiener measure, and the second representation is the energy representation. We define these representations and their basic properties, and then we prove that these representations are unitarily equivalent.
\end{abstract}

\maketitle

\tableofcontents

\renewcommand{\contentsname}{Table of Contents}

\section{Introduction}\label{s.1}

The main subject of this paper is a study of two unitary representations of the group $H\left( G \right)$ of smooth paths in a compact Lie group $G$. The first representation is on the Hilbert space $L^{2}\left( W\left( G \right), \mu \right)$, where $W\left( G \right)$ is the Wiener space of continuous path in $G$ and $\mu$ is the corresponding Wiener measure. This representation is induced by the quasi-invariance of the Wiener measure $\mu$ with respect to the left (right) multiplication on $W\left( G \right)$ by elements in $H\left( G \right)$. The necessary preliminaries from stochastic analysis are introduced in Section \ref{s.2}. We define the corresponding Brownian representations in Section \ref{s.4}. One of the questions mentioned in the previous works such as \cite{AlbeverioHKTestardVershik1983a} is whether the constant function $\mathbbm{1}$ is the cyclic vector for these representations. This is what we prove in Section \ref{s3}.

Another representation of the the group $H\left( G \right)$ is the energy representation. The representation space in this case is $L^{2}\left( W\left( \mathfrak{g} \right), \nu \right)$, where $\mathfrak{g}$ is the Lie algebra of $G$, and $\nu$ is the standard Gaussian measure on $W\left( \mathfrak{g}\right)$. Our main result in Section \ref{s.energy} is the (unitary) equivalence of the Brownian and energy representations.

These representations have been studied previously in a number of articles including \cite{AHKMTTBook, AlbeverioHK1978, AlbeverioHKTestardVershik1983a, AlbeverioHKTestardVershik1983b, DriverHall1999a, GGV1977, GGV1980-81, Testard1987, Wallach1987}. We will not attempt to give a comprehensive review of the mathematical literature on the subject, but rather explain the choice of this particular topic for this volume.

\subsection*{Acknowledgement} Even though M.I. had no publications in this field, the combination of representation theory, stochastic analysis and von Neumann algebras appealed to him. Moreover, he introduced MG to the latter subject which resulted in \cite{Gordina2002}.

\section{Notation}

\label{s.2}

Let $G$ be a compact connected Lie group, $e\in G$ denote the identity of $G,$
$\mathfrak{g}$ be its Lie algebra, and $d=\dim_{\mathbb{R}}\mathfrak{g}$ be
the dimension of $G$ and $\mathfrak{g}$. Without loss of generality we may and
do assume that $G$ is a Lie subgroup of $\mathrm{GL}_{n}(\mathbb{R})$. By
identifying $G$ with a matrix group, we are able to minimize the differential
geometric notation required of the reader. We assume that the Lie algebra
$\mathfrak{g}$ of $G$ is identified with the tangent space at $e$, and
$\mathfrak{g}$ is equipped with an $\operatorname{Ad}_{G}$-invariant inner
product $\langle\cdot,\cdot\rangle$, which we could take to be the negative of
the Killing form if $\mathfrak{g}$ is semi-simple. Associated to the
$\operatorname{Ad}_{G}$-invariant inner product is the Laplace operator
described below.

\subsection{Heat kernels}

This section reviews some basic facts about heat kernels on unimodular Lie groups. Let $dx$ denote a bi-invariant Haar measure on $G$ which is unique up to normalization. For $A\in\mathfrak{g}$, let $\widetilde{A} (\widehat{A})$
denote the unique left (right) invariant vector field on $G$ which agrees with $A$ at $e\in G$. Let $\mathfrak{g}_{0}\subset\mathfrak{g}$ be an orthonormal basis for $\mathfrak{g}$. The left and right invariant Laplacian is then given $\Delta:=\sum_{A\in\mathfrak{g}_{0}}\widetilde{A}^{2}$ and $\Delta^{\prime}
:=\sum_{A\in\mathfrak{g}_{0}}\widehat{A}^{2}$ respectively. Since $G$ is unimodular, it is easy to check the formal adjoint, relative to $\mathrm{L}^{2}\left(  G, dx \right)  $, of $\widetilde{A}$ $(\widehat{A})$ is $-\widetilde{A}$ $(-\widehat{A})$. Hence, $\Delta/2$ and $\Delta^{\prime}/2$ are symmetric operators on the smooth functions with compact support on $G$. It is well known, see for example Robinson \cite[Theorem 2.1, p. 152]{RobinsonBook1991}, that $\Delta/2$ and $\Delta^{\prime}/2$ are essentially self-adjoint and the closures of $\Delta/2$ and $\Delta^{\prime}/2$ generate strongly continuous, self-adjoint contraction semigroups
$e^{t\Delta/2}$ and $e^{t\Delta^{\prime}/2}$ on $\mathrm{L}^{2}\left(  G, dx \right)$. Let $p_{t}=e^{t\Delta/2}\delta_{e}$, $t \geqslant 0$, be the fundamental solution, i.e.

\begin{equation}
\partial p_{t}/\partial t=\frac{1}{2}\Delta p_{t}\mathrm{\ with}\text{ }
\lim_{t\rightarrow0}p_{t}=\delta_{e}. \label{e.k2.1}%
\end{equation}
For a proof of the following theorem see Robinson \cite[Theorem 2.1,
p.257]{RobinsonBook1991}.

\begin{thm}
\ \label{t.k2.7} Assuming the above notation, let $p_{t}$ denote the
fundamental solution to the left heat equation \text{(\ref{e.k2.1})}.
Then\ $p_{t}(x)=p_{t}(x^{-1})$ for all $x\in G$ and
\[
e^{t\Delta/2}f(x)=\int_{G}p_{t}(x^{-1}h)f(h)dh=\int_{G}p_{t}(h^{-1}x)f(h)dh.
\]

\end{thm}

\begin{ex}
\label{ex.flat} In the case we take $G$ to be $\mathfrak{g}$ thought of as a
Lie group with its additive structure, we recover the standard convolution
heat kernel relative to the Lebesgue measure given by
\[
p_{t}\left(  x\right)  =\left(  \frac{1}{2\pi t}\right)^{d/2}\exp\left(
-\frac{1}{2t}\left\vert x\right\vert_{\mathfrak{g}}^{2}\right)  .
\]

\end{ex}

\subsection{Wiener Measures}

The reader is referred to \cite[p. 502]{Shigekawa1984a}, \cite[Theorem
1.4]{MalliavinMalliavin1990a}, \cite{Driver1992b, Driver1994b} and perhaps
also in \cite{Driver1995a} for more details on the summary presented here.

\begin{nota}
\label{n.2.1}Suppose $0<T<\infty$. Let us introduce the Wiener and
Cameron-Martin (finite energy) spaces, and the corresponding probability measures.

\begin{enumerate}
\item \textbf{Wiener space} will refer to the continuous path space
\[
W(G)=W([0,T],G)=\{\gamma\in C([0,T],G):\gamma_{0}=e\},
\]
where we equip $W\left(  G\right)  $ with the uniform metric
\[
d_{\infty}\left(  \alpha,\beta\right)  :=\max_{t\in\left[  0,T\right]
}d\left(  \alpha_{t},\beta_{t}\right)  .
\]
Here $d$ is the left invariant metric on $G$ associated to the left invariant
Riemannian metric on $G$ induced from the $\operatorname{Ad}_{G}$--invariant
inner product $\langle\cdot,\cdot\rangle$ on $\mathfrak{g}$. [In fact, these
metrics are bi-invariant, i.e. both left and right invariant.] Let
$g_{t}:W\left(  G\right)  \rightarrow G$ (for $0\leqslant t\leqslant T)$ be
the projection maps defined by
\[
g_{t}\left(  \gamma\right)  :=\gamma_{t},\text{ for all }\gamma\in W\left(
G\right)  .
\]
We further make $W(G)$ into a group using pointwise multiplication by $\left(
hk\right)  _{t}:=h_{t}k_{t}$ for all $h,k\in W\left(  G\right)  )$ and
$\Theta:W\left(  G\right)  \rightarrow W\left(  G\right)  $ be the group
inversion defined by
\[
\Theta\left(  \gamma\right)  =\gamma^{-1}\text{ for all }\gamma\in W\left(
G\right)  .
\]

\item Given $h\in W\left(  G\right) $, let
\[
\Vert h\Vert_{H,T}^{2}=\left\{
\begin{array}
[c]{ll}%
\infty, & \text{if $h$ is not absolutely continuous,}\\
\int_{0}^{T}|h(s)^{-1}h^{\prime}(s)|^{2}ds, & \text{if $h$ is absolutely
continuous}.
\end{array}
\right.
\]
Here $|\cdot|$ is the norm induced by the inner product $\langle\cdot,
\cdot\rangle$ on the Lie algebra $\mathfrak{g}$.

\item The \textbf{Cameron-Martin} (\textbf{finite energy}) subgroup, $H\left(
G\right)  \subset W\left(  G\right) $, is defined by
\[
H(G)=\left\{  h\in W(G):\Vert h\Vert_{H,T}<\infty\right\} .
\]

\item The corresponding spaces of paths with values in the Lie algebra
$\mathfrak{g}$ and starting at $0$ are denoted by $W\left(  \mathfrak{g}
\right) $, and $H\left(  \mathfrak{g}\right) $, and the Wiener measure on
$W\left(  \mathfrak{g}\right) $ is denoted by $\nu$.
\end{enumerate}
\end{nota}

\begin{thm}
[Wiener measures]\label{t.2.3} Let $\mathcal{B}$ be the Borel $\sigma
$--algebra on $W\left(  G\right)  $. There is a probability measure $\mu$ on
$(W(G),\mathcal{B})$ uniquely determined by specifying its finite dimensional
distributions as follows. For all $k\in\mathbb{N}$, partitions $0=s_{0}%
<s_{1}<s_{2}<\ldots<s_{k-1}<s_{k}=T$ of $\left[  0, T\right]  $, and for all
bounded measurable functions $f:G^{k}\rightarrow\mathbb{R}$%

\begin{equation}
\mu(f(g_{s_{1}},\ldots,g_{s_{k}}))=\int_{G^{k}}f(x_{1},\ldots,x_{k}
)\prod_{i=1}^{k}p_{\Delta s_{i}}(x_{i-1}^{-1}x_{i})dx_{1}\cdots dx_{k},
\label{e.2.5}%
\end{equation}
where $x_{0}:=e$, $\Delta s_{i}\equiv s_{i}-s_{i-1}$, $p_{t}(x)$ is the
convolution heat kernel described in Theorem \ref{t.k2.7}.

The process, $\{g_{t}\}_{0\leqslant t\leqslant T}$, is a $G$--valued Brownian
motion with respect to the filtered probability space $(W(G),\{\mathcal{B}%
_{t}\}, \mathcal{B}, \mu)$. In more detail, $\{g_{t}\}_{0\leqslant t\leqslant
T}$ is a diffusion process on $G$ with generator $\frac{1}{2}\Delta$ such that
$g_{0}=e$ a.s. As usual, this process has the following martingale property:
for all $f\in(C^{\infty}(G))$ the process
\begin{equation}
M_{t}^{f}:=f(g_{t})-f(g_{0})-\frac{1}{2}\int_{0}^{t}\Delta f(g_{\tau}%
)d\tau\label{e.2.7}%
\end{equation}
is a local martingale. In differential form this can be written as
\begin{equation}
\label{e.2.4}df\left(  g\right)  \overset{m}{=}\frac{1}{2}\left(  \Delta
f\right)  \left(  g\right)  dt,
\end{equation}
where $da\overset{m}{=}db$ if $a-b$ is a local martingale.
\end{thm}

\begin{proof}
Equation (\ref{e.2.7}) is well known from the theory of Markov processes, see
\cite{StroockVaradhanBook2006}. Indeed, using the Markovian property of $\mu$
one computes for $s>t$, $F$ a bounded $\mathcal{B}_{t}$-measurable function,
and $f\in C^{\infty}\left(  G\right)  $%
\begin{align*}
\frac{d}{ds}\mu(f(g_{s})F)  &  =\frac{d}{ds}\mu((e^{\frac{s-t}{2}\Delta
}f)(g_{t})F)\\
&  =\frac{1}{2}\mu(e^{\frac{s-t}{2}\Delta}\Delta f)(g_{t})F)=\mu(\frac{1}%
{2}\Delta f(g_{s})F).
\end{align*}
Integrating the last expression from $t$ to $s$ shows that
\[
\mu([M_{t}^{f}-M_{s}^{f}]F)=\mu\left(  \left\{  f(g_{t})-f(g_{s})-\int_{s}%
^{t}\frac{1}{2}\Delta f(g_{\tau})d\tau\right\}  F\right)  =0,
\]
which shows that $M^{f}$ is a martingale.
\end{proof}

\begin{rem}
\label{r.2.5} Note that the martingale property \eqref{e.2.5} can be extended
to vector-valued function. In particular, this applies to $G$-valued functions
since $G$ is assumed to be a matrix-valued Lie group.
\end{rem}

\subsection{Left and right Brownian motions}

\begin{thm}
[Quadratic variations]\label{t.qv} If $u$ and $v$ are smooth functions on $G$ then

\[
d\left[  u\left(  g_{t}\right)  \right]  \cdot d\left[  v\left(  g_{t}\right)
\right]  =dM_{t}^{u}dM_{t}^{v}=\left(  \nabla u\left(  g_{t}\right)
\cdot\nabla v\left(  g_{t}\right)  \right)  dt=\sum_{A\in\mathfrak{g}_{0}%
}\left(  \widetilde{A}u\right)  \left(  g_{t}\right)  \widetilde{A}v\left(
g_{t}\right)  dt.
\]
In particular,
\[
dg_{t}\otimes dg_{t}=g_{t}\otimes g_{t}Cdt,
\]
where $C:=\sum_{A\in\mathfrak{g}_{0}}A\otimes A$.
\end{thm}

\begin{proof}
On one hand,%
\[
d\left[  uv(g)\right]  \overset{m}{=}\frac{1}{2}\Delta\left(  uv\right)
\left(  g\right)  dt=\frac{1}{2}\left(  \Delta uv+u\Delta v+2\nabla
u\cdot\nabla v\right)  (g_{t})dt
\]
while on the other by It\^{o}'s formula,%
\begin{align*}
d\left[  u\left(  g\right)  v\left(  g\right)  \right]  = &  du\left(
g\right)  \cdot v\left(  g\right)  +u\left(  g\right)  \cdot dv\left(
g\right)  +du\left(  g\right)  dv\left(  g\right)  \\
\overset{m}{=} &  \frac{1}{2}\left(  \Delta uv+u\Delta v\right)
(g_{t}dt+dM^{u}dM^{v}%
\end{align*}
Comparing these two equations shows
\[
dM^{u}dM^{v}\overset{m}{=}\left(  \nabla u\cdot\nabla v\right)  (g_{t})dt
\]
which gives the first result. More generally, suppose that $u$ and $v$ are
vector valued, then
\[
d\left[  u\otimes v\right]  (g)\overset{m}{=}\frac{1}{2}\Delta\left(  u\otimes
v\right)  \left(  g\right)  dt=\frac{1}{2}\left(  \Delta u\otimes
v+u\otimes\Delta v+2\widetilde{A}u\otimes\widetilde{A}v\right)  (g_{t})dt
\]
while on the other hand by It\^{o}'s formula,
\begin{align*}
d\left[  u\left(  g\right)  \otimes v\left(  g\right)  \right]   &  =d\left[
u\left(  g\right)  \right]  \otimes v\left(  g\right)  +u\left(  g\right)
\otimes d\left[  v\left(  g\right)  \right]  +d\left[  u\left(  g\right)
\right]  \otimes d\left[  v\left(  g\right)  \right]  \\
\overset{m}{=} &  \frac{1}{2}\left(  \Delta u\otimes v+u\otimes\Delta
v\right)  (g_{t})dt+dM^{u}\otimes dM^{v}
\end{align*}
Comparing these two equations shows
\[
dM^{u}\otimes dM^{v}\overset{m}{=}\sum_{A\in\mathfrak{g}_{0}}\left(
\widetilde{A}u\otimes\widetilde{A}v\right)  (g_{t})dt.
\]
By Remark \ref{r.2.5} we can take $u\left(  g\right)  =g$ and $v\left(
g\right)  =g$ to see that
\begin{equation}
dg_{t}\otimes dg_{t}=\sum_{A\in\mathfrak{g}_{0}}gA\otimes gAdt\label{e.ma}%
\end{equation}
and $dg=dM+\frac{1}{2}gCdt$, where $C=\sum_{A\in\mathfrak{g}_{0}}A^{2}$.
\end{proof}

\begin{rem}
Note that $C$ is independent of the choice of the orthonormal basis of
$\mathfrak{g}$ as was pointed out in \cite[Lemma 3.1]{Gordina2003a}.
\end{rem}

\begin{df}
[Left and right Brownian motions]\label{d.lr} The process $\left\{
g_{t}\right\} _{0\leqslant t\leqslant T}$ is a semi-martingale and therefore
we may define two $\mathfrak{g}$--valued processes by
\[
B_{t}^{L}:=\int_{0}^{t}g_{\tau}^{-1}\delta g_{\tau}\text{ and }B_{t}^{R}%
:=\int_{0}^{t}\delta g_{\tau}g_{\tau}^{-1}.
\]
We refer to $B^{L}$ ($B^{R})$ as the left (right) Brownian motion associated
to $\left\{  g_{t}\right\}  _{0\leqslant t\leqslant T}$. The terminology will
be justified by the next theorem.
\end{df}

\begin{thm}
\label{t.bms} $B_{t}^{L}:=\int_{0}^{t}g_{\tau}^{-1}\delta g_{\tau}$ and
$B_{t}^{R}:=\int_{0}^{t}\delta g_{\tau}g_{\tau}^{-1}$ are standard
$\mathfrak{g}$--valued Brownian motions with covariances determined by
$\left\langle \cdot,\cdot\right\rangle _{\mathfrak{g}}$.
\end{thm}

\begin{proof}
Let $b_{t}:=B_{t}^{L}:=\int_{0}^{t}g_{\tau}^{-1}\delta g_{\tau}$ temporarily.
Then
\begin{align*}
db  &  =g^{-1}\delta g=g^{-1}dg+\frac{1}{2}d\left[  g^{-1}\right]  dg\\
&  =g^{-1}dM+\frac{1}{2}Cdt-\frac{1}{2}dbg^{-1}dg\\
&  =g^{-1}dM+\frac{1}{2}Cdt-\frac{1}{2}dbdb
\end{align*}
but $dbdb=\left(  g^{-1}dg\right)  \left(  g^{-1}dg\right)  =Cdt$ from
\eqref{e.ma}. This shows $db$ is a martingale and that
\begin{align*}
db\otimes db  &  =\left(  g^{-1}\otimes g^{-1}\right)  dg\otimes dg\\
&  =\left(  g^{-1}\otimes g^{-1}\right)  \sum_{A\in\mathfrak{g}_{0}}gA\otimes
gAdt\\
&  =\sum_{A\in\mathfrak{g}_{0}}A\otimes Adt,
\end{align*}
and so by L\'{e}vy's criterion $b$ is a standard $\mathfrak{g}$--valued
Brownian motion. We now call $b=B^{L}$.
\end{proof}

\begin{thm}
\label{t.g-valued}Let $\varphi\in H\left(  G\right) $. The processes $\left\{
B_{t}^{L}\right\}  _{0\leqslant t\leqslant T}$ and~$\left\{  B_{t}
^{R}\right\}_{0\leqslant t\leqslant T}$ are $\mathfrak{g}$--valued Brownian
motions satisfying the following properties

\begin{enumerate}
\item $dB_{t}^{R}=\operatorname{Ad}_{g_{t}}dB_{t}^{L}=\operatorname{Ad}%
_{g_{t}}\delta B_{t}^{L}$,

\item $dB_{t}^{L}=\operatorname{Ad}_{g_{t}^{-1}}dB_{t}^{R}=\operatorname{Ad}%
_{g_{t}^{-1}}\delta B_{t}^{R}$,

\item $B_{t}^{L}\left(  \varphi^{-1}g\right)  =B_{t}^{L}-\int_{0}%
^{t}\operatorname{Ad}_{g^{-1}} \left(  \delta\varphi\varphi^{-1}\right) $,

\item $B_{t}^{L}\left(  g\varphi\right)  =\int_{0}^{t}\operatorname{Ad}%
_{\varphi^{-1}}dB^{L}+\int_{0}^{t}\varphi^{-1}\delta\varphi$

\item $B_{t}^{R}\left(  \varphi^{-1}g\right)  =-\int_{0}^{t}\varphi^{-1}%
\delta\varphi+\int_{0}^{t}\operatorname{Ad}_{\varphi^{-1}}\delta B^{R}$

\item $B_{t}^{R}\left(  g\varphi\right)  =B_{t}^{R}+\int_{0}^{t}%
\operatorname{Ad}_{g}\left(  \delta\varphi\varphi^{-1}\right) $.
\end{enumerate}
\end{thm}

\begin{proof}
For (1)
\begin{align*}
\delta B_{t}^{R}  &  =\delta g_{t}g_{t}^{-1}=g_{t}g_{t}^{-1}\delta g_{t}%
g_{t}^{-1}=\operatorname{Ad}_{g_{t}}\delta B_{t}^{L}\\
&  =\operatorname{Ad}_{g_{t}}dB_{t}^{L}+\frac{1}{2}\left(  d\left[
\operatorname{Ad}_{g_{t}}\right]  \right)  dB_{t}^{L}\\
&  =\operatorname{Ad}_{g_{t}}dB_{t}^{L}+\frac{1}{2}\operatorname{Ad}%
_{g}ad_{dB^{L}}dB_{t}^{L}=\operatorname{Ad}_{g_{t}}dB_{t}^{L},
\end{align*}
where we have used the fact that $\delta g=g\delta B^{L}$ implies
$\delta\operatorname{Ad}_{g}=\operatorname{Ad}_{g}ad_{\delta B^{L}}$, and
$ad_{dB^{L}}dB_{t}^{L}=0$. For (2)
\begin{equation}
dB_{t}^{L}=g_{t}^{-1}\delta g_{t}=g_{t}^{-1}\left[  \delta g_{t}g_{t}
^{-1}\right]  g_{t}=\operatorname{Ad}_{g_{t}^{-1}}\delta B_{t}^{R}%
.\label{e.twine}%
\end{equation}
Since $\delta g=\delta B^{R}g$ implies that $\delta g^{-1}=-g^{-1}\delta
B^{R}$, and therefore $\delta\operatorname{Ad}_{g^{-1}}=-\operatorname{Ad}%
_{g^{-1}}ad_{\delta B^{R}}$, and so the It\^{o} form of \eqref{e.twine} is%

\begin{align*}
dB_{t}^{L}  &  =\operatorname{Ad}_{g_{t}^{-1}}dB_{t}^{R}+\frac{1}{2}\left(
d\left[  \operatorname{Ad}_{g_{t}^{-1}}\right]  \right)  dB_{t}^{R}\\
&  =\operatorname{Ad}_{g_{t}^{-1}}dB_{t}^{R}-\frac{1}{2}\operatorname{Ad}%
_{g^{-1}}ad_{\delta B^{R}}dB_{t}^{R}=\operatorname{Ad}_{g_{t}^{-1}}dB_{t}^{R}.
\end{align*}
The remaining items, (3--6), follow from simple computations in It\^{o}'s
calculus
\begin{align*}
B_{t}^{L}\left(  \varphi^{-1}g\right)   &  =\int_{0}^{t}\left(  \varphi
^{-1}g\right) ^{-1}\delta\left(  \varphi^{-1}g\right)  =\int_{0}^{t}%
g^{-1}\varphi\left(  -\varphi^{-1}\delta\varphi\varphi^{-1}g+\varphi
^{-1}\delta g\right) \\
&  =B_{t}^{L}-\int_{0}^{t}\operatorname{Ad}_{g^{-1}}\left(  \delta
\varphi\varphi^{-1}\right) ,\\
B_{t}^{L}\left(  g\varphi\right)   &  =\int_{0}^{t}\left(  g\varphi\right)
^{-1}\delta\left(  g\varphi\right)  =\int_{0}^{t}\varphi^{-1}g^{-1}\left(
\delta g\varphi+g\delta\varphi\right) \\
&  =\int_{0}^{t}\operatorname{Ad}_{\varphi^{-1}}dB^{L}+\int_{0}^{t}%
\varphi^{-1}\delta\varphi,\\
B_{t}^{R}\left(  \varphi^{-1}g\right)   &  =\int_{0}^{t}\delta\left(
\varphi^{-1}g\right)  \left(  \varphi^{-1}g\right)  ^{-1}=\int_{0}^{t}\left(
-\varphi^{-1}\delta\varphi\varphi^{-1}g+\varphi^{-1}\delta g\right)
g^{-1}\varphi\\
&  =-\int_{0}^{t}\varphi^{-1}\delta\varphi+\int_{0}^{t}\operatorname{Ad}%
_{\varphi^{-1}}\delta B^{R},\text{ and}\\
B_{t}^{R}\left(  g\varphi\right)   &  =\int_{0}^{t}\delta\left(
g\varphi\right)  \left(  g\varphi\right)  ^{-1}=\int_{0}^{t}\left(  \delta
g\varphi+g\delta\varphi\right)  \varphi^{-1}g^{-1}\\
&  =B_{t}^{R}+\int_{0}^{t}\operatorname{Ad}_{g}\left[  \delta\varphi
\varphi^{-1}\right]  .
\end{align*}

\end{proof}

Before introducing It\^{o} maps, recall some standard definitions.

\begin{nota}
\label{n.3.2} Suppose $\left(  X, \mathcal{B}, \mu\right) $ is a measurable
space with a $\sigma$-finite Borel measure $\mu$, and $R$ is a measurable
bijection on $X$. Then the pushforward of $\mu$ is defined by%

\[
\left(  R_{\ast}\mu\right)  \left(  A\right)  :=\left(  \mu\circ
R^{-1}\right)  \left(  A\right)  =\mu\left(  R^{-1}\left(  A\right)  \right)
,~A\in\mathcal{B}.
\]
If the pushforward measure $R_{\ast}\mu$ is equivalent to $\mu$, we will
denote the Radon-Nikodym derivative as usual by%

\[
\frac{dR_{\ast}\mu}{d\mu}\left(  x\right)  ,~x\in X.
\]

\end{nota}

In particular, for any $A \in\mathcal{B} \left(  X \right)  $ we have%

\[
\int_{X} \mathbbm{1}_{A}\left(  x \right)  dR_{\ast}\mu=\int_{X}
\mathbbm{1}_{R^{-1}\left(  A \right)  }\left(  x \right)  d\mu=\int_{X}
\mathbbm{1}_{ A }\left(  R\left(  x \right)  \right)  d\mu.
\]

\begin{nota}
\label{n.4.3} Let $\left(  X,\mathcal{Q}_{1}\right)  $, $\left(
Y,\mathcal{Q}_{2}\right)  $ be two measurable spaces, and let $I:X\rightarrow
Y$ be a measurable map. Then for any measurable function $f:Y\rightarrow \mathbb{R}$ we denote by
%\marginpar{$I_{\ast}f$ would be better denote by
%$I^{\ast}f$ instead as it is a contravariant functor, i.e. it pulls functions
%back. Is this notation actually used below?}
\[
\left(  I^{\ast}f\right)  \left(  x\right)  :=f\left(  I\left(  x\right)
\right)
\]
the induced map on the set of measurable functions on $X$.
\end{nota}

\begin{prop}
\label{p.3.3} The maps $B^{L},B^{R}:\left(  W\left(  G\right), \mu\right)  \rightarrow\left(  W\left(
\mathfrak{g}\right), \nu\right)$
%\marginpar{I don't see why we need to define the maps $\mathcal{I}_{L}$ and $\mathcal{I}_{R}.$ I do not
%think they are really used, just their inverses which have already been defined as $B^{L}$ and $B^{R}$ %respectively. Moreover, I think the given definitions are confusing.  I would just say that the B's are measure
%theoretic isomorphism and concentrate on Eq. (\ref{e.3a.8}).}
are $\mu$--a.e. defined maps such that $B_{\ast}^{L}\mu=\nu=B_{\ast}^{R}\mu$.
In fact, these maps are measure-preserving isomorphisms from $\left(  W\left(
G\right)  ,\mu\right)  $ to $\left(  W\left(  \mathfrak{g}\right), \nu\right)$ with the inverse maps given by solving the SDEs
\[
\delta w=w\delta B^{L}\text{ or }\delta w=\delta B^{R}w\text{ with }w_{0}=e
\]
for $w$. Moreover, we have the identities
\begin{equation}
B^{L}\circ\Theta=-B^{R}\text{ a.e. and }B^{R}\circ\Theta=-B^{L}\text{
a.e.},\label{e.3a.8}%
\end{equation}
where the inversion map $\Theta$ is defined in Notation \ref{n.2.1}.
\end{prop}

\begin{proof}
Since
\[
\delta g=\delta B^{R}g\implies\delta g^{-1}=-g^{-1}\delta B^{R}%
\]
and hence%
\[
B^{L}\circ\Theta=B^{L}\circ\Theta\left(  g\right)  =\int_{0}^{\cdot}\left(
g^{-1}\right)  ^{-1}\delta g^{-1}=\int_{0}^{\cdot}g\left(  -g^{-1}\delta
B^{R}\right)  =\int_{0}^{\cdot}-\delta B^{R}=-B^{R}.
\]
Similarly one shows $B^{R}\circ\Theta=-B^{L}$ a.e.
\end{proof}

Note that the maps $B^{L}$ and $B^{R}$ induce maps on measurable functions from $\left(  W\left(  G\right), \mu\right)$ to $\left( W\left(  \mathfrak{g}\right), \nu\right)$ as described in Notation \ref{n.4.3}.

\subsection{Quasi-invariance\label{s.quasi}}

Our goal in this section is to understand the quasi-invariance properties of
$\mu$ under left and right translations by $\varphi\in H\left(  G\right) $.

\begin{thm}
\label{t.quasi} For $\varphi\in H\left(  G\right)  $ let
\[
Z_{T}^{R}\left(  \varphi\right)  : =\exp\left(  -\int_{0}^{T}\left\langle
\varphi^{\prime}\varphi^{-1}, \delta B^{L}\right\rangle -\frac{1}{2}\int
_{0}^{T}\left\vert \varphi^{\prime}\varphi^{-1} \right\vert ^{2}dt\right)
\]
and
\[
Z_{T}^{L}\left(  \varphi\right)  : =\exp\left(  \int_{0}^{T}\left\langle
\varphi^{\prime}\varphi^{-1},\delta B^{R}\right\rangle -\frac{1}{2}\int
_{0}^{T}\left\vert \varphi^{\prime}\varphi^{-1}\right\vert ^{2}dt\right)
\]
then
\[
\operatorname*{Law}\nolimits_{Z_{T}^{R}\cdot\mu}\left(  g\varphi\right)
=\operatorname*{Law}\nolimits_{\mu}\left(  g\right)  =\operatorname*{Law}%
\nolimits_{Z_{T}^{L}\cdot\mu}\left(  \varphi^{-1}g\right)  .
\]
That is, for every bounded and measurable function $F$ on $W\left(  G\right)
$%

\[
\int_{W\left(  G\right)  }F\left(  g\varphi\right)  Z_{T}^{R}\left(
\varphi\right)  d\mu=\int_{W\left(  G\right)  }Fd\mu=\int_{W\left(  G\right)
}F\left(  \varphi^{-1}g\right)  Z_{T}^{L}\left(  \varphi\right)  d\mu.
\]

\end{thm}

\begin{proof}
We will only prove the assertion involving the right translation here as the
second case is proved similarly. To simplify notation let $b:=B^{L}$,
\[
M_{t}:=-\int_{0}^{t}\left\langle \varphi^{\prime}\varphi^{-1},\delta
b\right\rangle =-\int_{0}^{t}\left\langle \varphi^{\prime}\varphi
^{-1},db\right\rangle
\]
and let $Z$ solve
\begin{equation}
dZ=ZdM=-Z\left\langle \varphi^{\prime}\varphi^{-1}, db\right\rangle \text{
with }Z_{0}=1, \label{e.zz}%
\end{equation}
i.e.%
\[
Z_{t}:=\exp\left(  -\int_{0}^{t}\left\langle \varphi^{\prime}\varphi
^{-1},\delta b\right\rangle -\frac{1}{2}\int_{0}^{t}\left\vert \varphi
^{\prime}\varphi^{-1}\right\vert ^{2}dt\right)  =Z_{t}^{R}\left(
\varphi\right)  .
\]
By (4) of Theorem \ref{t.g-valued}
\[
\left(  g\varphi\right) ^{-1}\delta\left(  g\varphi\right)  =\operatorname{Ad}%
_{\varphi^{-1}}\delta b+\varphi^{-1}d\varphi.
\]
So given a smooth function, $f:G\rightarrow\mathbb{R}$, we have by It\^{o}'s
lemma that
\begin{equation}
\delta\left(  f\left(  g\varphi\right)  \right)  =f^{\prime}\left(
g\varphi\right)  \left(  \operatorname{Ad}_{\varphi^{-1}}\delta b+\varphi
^{-1}d\varphi\right) ,\label{e.fg}%
\end{equation}
where for $A, B \in\mathfrak{g}$
\begin{align*}
f^{\prime}\left(  g\right)  A  &  =\widetilde{A}f\left(  g\right)  :=\left.
\frac{d}{dt}\right| _{0}f\left(  g e^{tA}\right)  \text{ and }\\
f^{\prime\prime}\left(  g\right)  \left[  A\otimes B\right]   &  :=\left.
\left. \left(  \widetilde{A}\widetilde{B}f\right)  \left(  g\right)  =\frac{d}
{dt}\right| _{0}\frac{d}{ds}\right| _{0}f\left(  g e^{tA}e^{sB}\right) .
\end{align*}
Note that
\begin{align*}
f^{\prime}\left(  g\varphi\right)  \operatorname{Ad}_{\varphi^{-1}}\delta b
&  =f^{\prime}\left(  g\varphi\right)  \operatorname{Ad}_{\varphi^{-1}%
}db+\frac{1}{2}d\left[  f^{\prime}\left(  g\varphi\right)  \right]
\operatorname{Ad}_{\varphi^{-1}}db\\
&  =f^{\prime}\left(  g\varphi\right)  \operatorname{Ad}_{\varphi^{-1}%
}db+\frac{1}{2}\left[  f^{\prime\prime}\left(  g\varphi\right)  \right]
\left[  \operatorname{Ad}_{\varphi^{-1}}db\otimes\operatorname{Ad}%
_{\varphi^{-1}}db\right] \\
&  =f^{\prime}\left(  g\varphi\right)  \operatorname{Ad}_{\varphi^{-1}%
}db+\frac{1}{2}\Delta f\left(  g\varphi\right)  dt.
\end{align*}
Now we can use the fact that%

\begin{equation}
\int_{0}\operatorname{Ad}_{\varphi^{-1}}db\label{e.2.15}%
\end{equation}
is a $\mathfrak{g}$-valued Brownian motion by L\'{e}vy's criterion and due to
the $\operatorname{Ad}$-invariance of the inner product on $\mathfrak{g}$.
Then the It\^{o} form of \eqref{e.fg} is
\[
d\left[  f\left(  g\varphi\right)  \right]  =f^{\prime}\left(  g\varphi
\right)  \operatorname{Ad}_{\varphi^{-1}}db+\left[  f^{\prime}\left(
g\varphi\right)  \varphi^{-1}\varphi^{\prime}+\frac{1}{2}\Delta f\left(
g\varphi\right)  \right]  dt.
\]
So if we define
\[
N_{t}=N_{t}^{f}:=f\left(  g_{t}\varphi_{t}\right)  -\frac{1}{2}\int_{0}%
^{t}\Delta f\left(  g_{\tau}\varphi_{\tau}\right)  d\tau,
\]
then
\[
dN=f^{\prime}\left(  g\varphi\right)  \operatorname{Ad}_{\varphi^{-1}%
}db+f^{\prime}\left(  g\varphi\right)  \varphi^{-1}\varphi^{\prime}dt.
\]
Observe that using the orthonormal basis $\mathfrak{g}_{0}$ of the Lie algebra
$\mathfrak{g}$ we have (using $db\otimes db=\sum_{A\in\mathfrak{g}_{0}%
}A\otimes Adt)$ that
\begin{align*}
\left(  \operatorname{Ad}_{\varphi^{-1}}db\right)  \left\langle \varphi
^{\prime}\varphi^{-1},db\right\rangle  &  =\sum_{A\in\mathfrak{g}_{0}}\left(
\operatorname{Ad}_{\varphi^{-1}}A\right)  \left\langle \varphi^{\prime}%
\varphi^{-1},A\right\rangle dt\\
&  =\operatorname{Ad}_{\varphi^{-1}}\left(  \varphi^{\prime}\varphi
^{-1}\right)  dt=\varphi^{-1}\varphi^{\prime}dt.
\end{align*}
Another application of It\^{o}'s lemma then implies
\begin{align*}
d\left[  NZ\right]   &  =dNZ+NdZ+dNdZ\\
&  \overset{m}{=}Z\left[  f^{\prime}\left(  g\varphi\right)  \varphi
^{-1}\varphi^{\prime}dt\right]  -\left(  f^{\prime}\left(  g\varphi\right)
\operatorname{Ad}_{\varphi^{-1}}db\right)  \cdot Z\left\langle \varphi
^{\prime}\varphi^{-1},db\right\rangle \\
&  =Z\left[  f^{\prime}\left(  g\varphi\right)  \varphi^{-1}\varphi^{\prime
}dt\right]  -Z\left(  f^{\prime}\left(  g\varphi\right)  \operatorname{Ad}%
_{\varphi^{-1}}\varphi^{\prime}\varphi^{-1}\right)  dt=0,
\end{align*}
where as in \eqref{e.2.4} we write $dX\overset{m}{=}dY$ if $X$ and $Y$ are two
processes such that $Y-X$ is a martingale. The previous computations show $NZ$
is martingale and so%
\[
\mathbb{E}\left[  \left(  N_{t}-N_{s}\right)  FZ_{T}\right]  =0
\]
for all bounded $\mathcal{B}_{s}$--measurable functions $F$. Therefore
$\left\{  N_{t}^{f}\right\}  _{0\leqslant t\leqslant T}$ is a $Z_{T}\cdot\mu
$--martingale for all smooth $f$. Thus it follows from uniqueness to the
martingale problems that $\operatorname*{Law}_{Z_{T}\cdot\mu}\left(
g\varphi\right)  =\operatorname*{Law}_{\mu}\left(  g\right)  $.
\end{proof}

Theorem \ref{t.quasi} can be interpreted also using Notation \ref{n.3.2}.
Namely, for $X=W\left(  G\right)  $ and a measurable bijection $R$ on
$W\left(  G\right)  $ we have that for any Borel measurable $f$ on $W\left(
G\right)  $
\[
\mathbb{E}_{R_{\ast}\mu}f\left(  g\right)  =\mathbb{E}_{R_{\mu}}f\left(
R\left(  g\right)  \right).
\]
Let $L_{\varphi}, R_{\varphi}$ be the left and right multiplication on
$W\left(  G\right) $ defined by
%\marginpar{I am not a big fan of putting the inverse in the definition of $L_{\varphi}$ as I think it is not %standard, but it is ok. }
\begin{align}
L_{\varphi}g &  :={\varphi}^{-1}g,\nonumber\\
R_{\varphi}g &  :=g\varphi,\label{e.5.2}%
\end{align}
where $\varphi\in H\left(  G\right)$, and $g\in W\left(  G\right)  $,
together with their counterparts on functions on $W\left(  G\right)  $ denoted
by $L_{\varphi\ast}$ and $R_{\varphi\ast}$ according to Notation \ref{n.4.3}.
In addition, taking inverses in $\left(  W\left(  G\right), \mu\right)  $
induces a map on the set of measurable functions on $\left(  W\left(
G\right), \mu\right)  $ by
\begin{equation}
\left(  Jf\right)  \left(  \gamma\right)  :=f\circ\Theta\left(  \gamma\right)
=f\left(  \gamma^{-1}\right)  .\label{e.5.1}%
\end{equation}
Note that by Proposition \ref{p.3.3} the map $J$ is a unitary involution on
$L^{2}\left(  W\left(  G\right)  ,\mu\right)  $.

Then Theorem \ref{t.g-valued} can be re-written as follows. For any
$\varphi\in H\left(  G\right)  $ and $g\in W\left(  G\right)$ we
have
\begin{align}
B^{L}\left(  L_{\varphi}g\right)   &  =B^{L}\left(  g\right)
-\int_{0}^{\cdot}\operatorname{Ad}_{g^{-1}}\left(  d\varphi\varphi
^{-1}\right)  ,\nonumber\\
B^{L}\left(  R_{\varphi}g\right)   &  =\int_{0}^{\cdot}%
\varphi^{-1}d\varphi+\int_{0}^{\cdot}\operatorname{Ad}_{\varphi^{-1}}\left(
\delta B^{L}\right)  ,\nonumber\\
B^{R}\left(  L_{\varphi}g\right)   &  =-\int_{0}^{\cdot}%
\varphi^{-1}d\varphi+\int_{0}^{\cdot}\operatorname{Ad}_{\varphi^{-1}}\left(
\delta B^{R}\right)  ,\label{e.4.2}\\
B^{R}\left(  R_{\varphi}g\right)   &  =B^{R}\left(  g\right)
+\int_{0}^{\cdot}\operatorname{Ad}_{g}\left(  d\varphi\varphi^{-1}\right)
,\nonumber
\end{align}
where we use $d\varphi$ to indicate that it is the usual differential since
$\varphi$ is smooth.

Then the right Radon-Nikodym density $Z^{R}\left(  \varphi\right) $ for
$R_{\varphi\ast}\mu$ with respect to $\mu$ is in $L^{1}\left(  W(G),\mu\right)
$ is described in Theorem \ref{t.quasi}. Similarly the Wiener measure $\mu$ is
quasi-invariant under the left multiplication by elements in $H(G)$, and the
left Radon-Nikodym density for $\mu$ is in $L^{1}\left(  W(G),\mu\right) $ as well.

\begin{prop}
\label{p.3.2} The left and right Radon-Nikodym densities for $\mu$ satisfy
\[
Z^{R}_{\varphi}=JZ^{L}_{\varphi}=Z^{L}_{\varphi}\circ\Theta
\]
for $\mu$-almost every $g$. Here $J$ is the map defined by \eqref{e.5.1}.
\end{prop}

\begin{proof}
\textbf{First proof.} By Proposition \ref{p.3.3} $\mu$ is invariant under the
taking group inverses, that is, for any bounded measurable $f$
\[
\int_{W(G)}f(g^{-1})d\mu(g)=\int_{W(G)}f(g)d\mu(g).
\]
Then
\begin{align*}
&  \int_{W(G)}f(g\varphi)d\mu(g)=\int_{W(G)}f(g^{-1}\varphi)d\mu(g)=\\
&  \int_{W(G)}f\left(  \left(  \varphi^{-1}g^{-1}\right)  ^{-1}\right)
d\mu(g)=\int_{W(G)}f\left(  g^{-1}\right)  Z^{L}_{\varphi}(g)d\mu(g)\\
&  \int_{W(G)}f\left(  g\right)  Z^{L}_{\varphi}(g^{-1})d\mu(g).
\end{align*}

\end{proof}

\section{Cyclicity\label{s3}}

Cyclicty is one of the basic properties of representations of $H\left(  G \right)$ we consider later. Note that the main result of this section, Theorem \ref{t.4.11}, follows from Corollary 14 in \cite{HallSengupta1998}. In
that paper B. Hall and A. Sengupta used the Segal-Bargmann transform to prove the cyclicity of $\mathbbm{1}$, and also that the Radon-Nikodym densities are coherent states as Theorem 10 in \cite{HallSengupta1998} states. We give a more direct proof using the inverse It\^{o} map $B^{L}$ and ideas of L. Gross in \cite{Gross1993a}.

\begin{thm}
[Cyclicity of $\mathbbm{1}$]\label{t.4.11} Suppose that $G$ is a compact
connected Lie group, then
\[
\mathcal{H}_{G}:=\operatorname{Span}\left\{  \left(  Z_{\varphi}^{R}\left(
g\right)  \right)^{1/2}, \varphi\in H\left(  G\right)  \right\}
\]
is dense in $L^{2}\left(  W\left(  G\right), \mu\right)$.
\end{thm}

\begin{proof}
Note that $\left( B^{L}\right)^{\ast}\left( Z_{\varphi}^{R} \right)^{1/2}$ is a function on $W\left(  \mathfrak{g}\right)$ since $B^{L}$ is a measure space isomorphism, so we can reduce the problem to the Lie algebra level. Namely, let  $0=t_{0}<t_{1}<...<t_{n-1}<t_{n}=T$, $\xi_{0}=0,\xi_{1}, ..., \xi_{n}\in\mathfrak{g}$. We assume that

\begin{equation}
\label{e.4.15}\vert\xi_{j}\vert\vert t_{j}-t_{j-1}\vert=1, \text{ for any }
j=1, 2,...,n,
\end{equation}
unless $\xi_{j}=0$. It is known that the linear span of multidimensional
Hermite polynomials in $\langle\xi_{j}, w(t_{j})- w(t_{j-1})\rangle$ is dense
in $L^{2}\left(  W\left(  \mathfrak{g}\right)  , \nu\right)  $ (e.g.
\cite{Marion1989}). This means that it is enough to show that the linear span
of cylinder Hermite polynomials is contained in the $L^{2}\left(  W\left(
\mathfrak{g}\right), \nu\right)$-closure of $\left( B^{L}\right)^{\ast}\left(  \mathcal{H}_{G}\right)$.

First we observe that $\mathcal{H}_{G}$, and therefore $\left( B^{L}\right)^{\ast}\left(  \mathcal{H}_{G}\right)  $, contains all constant functions. Let
$0=t_{0}<t_{1}<...<t_{n-1}<t_{n}=T$, $\xi_{0}=0, \xi_{1}, ..., \xi_{n}
\in\mathfrak{g}$. We define a function $\varphi=\varphi_{\xi_{1}, ..., \xi
_{n}}\left(  s \right)  $ recursively for $j=1, 2,...,n$ by%

\begin{equation}
\label{e.4.16}\varphi\left(  t_{0}\right)  =\varphi\left(  0\right)  =e,
\ \varphi\left(  s\right)  =e^{-\left(  s-t_{j-1}\right)  \xi_{j}}%
\varphi\left(  t_{j-1}\right), \ s \in[ t_{j-1}, t_{j}).
\end{equation}
Then%

\[
\varphi^{\prime}(s)\varphi(s)^{-1}=-\xi_{j}, \ s \in[ t_{j-1}, t_{j} ),
\]
therefore $\varphi\in H\left(  G\right)  $ and%

\begin{align*}
&  \left( B^{L}\right)^{\ast} \left( Z^{R}_{\varphi}\right)  ^{1/2}\left(
w_{t}\right)  =\\
&  \prod_{j=1}^{n}\exp\left(  \frac{1}{2}\langle\xi_{j}, w(t_{j})-
w(t_{j-1})\rangle-\frac{1}{4}\vert\xi_{j} \vert^{2}\left(  t_{j}
-t_{j-1}\right)^{2}\right).
\end{align*}
Suppose $x_{1},...,x_{n} \in\mathbb{R}$ and define $\varphi_{\vec{x}%
}(s):=\varphi_{x_{1}\xi_{1},...,x_{n}\xi_{n} }( s)$, then $\varphi_{\vec{x}%
}^{\prime}(s)\varphi_{\vec{x}}(s)^{-1}=x_{j}\xi_{j}$. Now let a function $F$
on $\mathbb{R}^{n}$ be defined as $F\left(  \vec{x}\right):=\left(B^{L}\right)^{\ast}\left(  Z_{\varphi_{\vec{x}}}^{R}\right)^{1/2}$ then
\[
\frac{\partial F}{\partial x_{j}}\left(  0\right)  =\frac{1}{2} \langle\xi
_{j}, w(t_{j})- w(t_{j-1})\rangle,
\]
Note that for any $\vec{x} \in \mathbb{R}^{n}$ we have $F\left(  \vec {x}\right)  \in \left( B^{L}\right)^{\ast}\left(  \mathcal{H}_{G}\right)$. Therefore $\frac{\partial F}{\partial x_{j}}\left(  0\right)$ as well as all other partial derivatives of $F$ at $0$ are in $\overline{ \left( B^{L}\right)^{\ast} \left(  \mathcal{H}_{G}\right)  }$, the $L^{2}$-closure of $\left( B^{L}\right)^{\ast}\left(  \mathcal{H}_{G}\right)$. Indeed, this follows from the simple observation that $F\left(  0 \right)=1 \in \left( B^{L}\right)^{\ast}\left(  \mathcal{H}_{G}\right)  $ and

\[
\frac{\partial F}{\partial x_{j}}\left(  0\right)  =\lim_{x_{j} \to0}
\frac{F\left(  \left(  0,..., x_{j}, 0,...,0 \right)  \right)  -1}{x_{j}}.
\]
Now we would like to describe the functions we can get by taking partial
derivatives of $F$. First we observe that we can write $F$ as%

\[
F\left(  \vec{x} \right)  =\prod_{j=1}^{n} e^{a_{j}x_{j}-b_{j}^{2}x_{j}^{2}},
a_{j}=\frac{\langle\xi_{j}, w(t_{j})- w(t_{j-1})\rangle}{2}, b_{j}=\frac
{\vert\xi_{j} \vert\vert t_{j}-t_{j-1}\vert}{2}=\frac{1}{2}
\]
by assumption \eqref{e.4.15}. Using \cite[Lemma 1.3.2 (part (iii))]%
{BogachevBook} we can take partial derivatives of $F$ of all orders to see
that all multidimensional Hermite polynomials in $\langle\xi_{j}, w(t_{j})-
w(t_{j-1})\rangle$ are in $\overline{ \left( B^{L}\right)^{\ast}\left(
\mathcal{H}_{G}\right)  }$.
\end{proof}

\section{Brownian measure representation}

\label{s.4}

\subsection{Definitions and notation} The unitary representations of $H\left(  G\right)$ on the Hilbert space $L^{2}\left(  W\left( G\right), \mu\right)$ we define in this section are induced by quasi-invariance of the Wiener measure $\mu$. Recall that $L_{\varphi}$ and $R_{\varphi}$ are left and right multiplication on $W\left(  G\right)  $ by elements $H\left(  G\right)  $ as defined in \eqref{e.5.2}, i.e. $R_{\varphi }\gamma=\gamma\varphi$ and $L_{\varphi}\gamma=\varphi^{-1}\gamma$.

\begin{df}
\label{d.4.1} Let $W(G)$ and $H(G)$ be as before.

\begin{enumerate}
\item The \textbf{right Brownian measure representation} $U^{R}$ of $H(G)$ on
$L^{2}\left(  W(G),\mu\right)  $ is defined as%
\[
\left(  U_{\varphi}^{R}f\right)  \left(  g\right) : =\left(  Z_{\varphi}%
^{R}\left(  g\right)  \right) ^{1/2}f\left(  R_{\varphi}g\right)
\]
for any $f\in L^{2}\left(  W(G),\mu\right) $, $\varphi\in H(G)$, $g\in W(G)$;

\item the \textbf{left Brownian measure representation} $U^{L}$ on
$L^{2}\left(  W(G),\mu\right)  $ is defined as%

\[
\left(  U_{\varphi}^{L}f\right)  \left(  g\right) : =\left(  Z_{\varphi}%
^{L}\left(  g\right)  \right) ^{1/2}f\left(  L_{\varphi}g\right)
\]
for any $f\in L^{2}\left(  W(G), \mu\right) $, $\varphi\in H(G)$, $g\in W(G)$.
\end{enumerate}
\end{df}

Recall that by Proposition \ref{p.3.2} we have $Z^{R}_{\varphi}=J Z^{L}_{\varphi}$, where $J$ a unitary involution on $L^{2}\left(  W(G), \mu\right)$ defined by \eqref{e.5.1}. In addition, the functions $\left( Z^{R}_{\varphi} \right)^{1/2}$ and $\left(  Z^{L}_{\varphi} \right)^{1/2}$ have the norm $1$ in $L^{2}\left(  W(G), \mu\right)$ for any $\varphi, \psi\in H(G)$, which is a consequence of the next Proposition.

\begin{prop}
\label{p.5.3} For any $\varphi, \psi\in H\left(  G \right)  $%

\begin{multline*}
\langle\left(  Z_{\varphi}^{R}\right) ^{1/2}, \left(  Z_{\psi}^{R}\right)
^{1/2}\rangle=\\
\exp\left(  -\frac{\Vert\varphi\Vert_{H,T}^{2}+\Vert\psi\Vert_{H, T}^{2}}%
{8}\right)  \exp\left(  \frac{1}{4}\int_{0}^{T}\langle\left(  \varphi
^{-1}\varphi^{\prime}\right)  \left(  t\right) , \left(  \psi^{-1}\psi
^{\prime}\right)  \left(  t\right)  \rangle dt\right)  .
\end{multline*}

\end{prop}

\begin{proof}
This follows from Theorem \ref{t.quasi}.
\end{proof}

\begin{prop}
\label{p.5.4} For any $\varphi, \psi\in H\left(  I, G\right)  $ we have%

\[
Z^{R}_{\varphi} \left(  \cdot\right)  = Z^{R}_{ \psi} \left(  \cdot\right)
\text{ if and only if } \varphi=\psi,
\]
and similarly
\[
Z_{\varphi}^{L}\left(  \cdot\right)  = Z_{\psi}^{L}\left(  \cdot\right)
\text{ if and only if }\varphi=\psi,
\]
where $Z^{R}\left(  \varphi\right)  \left(  \cdot\right)  $ and $Z^{L}\left(
\varphi\right)  \left(  \cdot\right)  $ are viewed as random variables, and
the equalities hold for $\mu$-a.e. $g,t\in\lbrack0,T]$.
\end{prop}

\begin{proof}
If
\[
Z_{\varphi}^{R}\left(  \cdot\right)  = Z_{\psi}^{R}\left(  \cdot\right)  ,
\]
then for any $t\in\lbrack0, T]$,
\[
\mathbb{E}\left(  Z_{\varphi}^{R}\left(  \cdot\right)  |\mathcal{F}%
_{t}\right)  =\mathbb{E}\left(  Z_{\psi}^{R}\left(  \cdot\right)
|\mathcal{F}_{t}\right)
\]
and therefore
\[
\int_{0}^{t}\langle\psi^{-1}\psi^{\prime}(s)-\varphi^{-1}\varphi^{\prime}(s),
dB_{s}^{L}\rangle=\frac{1}{2}\int_{0}^{t}\left(  |\varphi^{-1}\varphi^{\prime
}|^{2}-|\psi^{-1}\psi^{\prime}|^{2}\right)  ds.
\]
Taking expectations of this equation then shows
\[
0=\frac{1}{2}\int_{0}^{t}\left(  |\varphi^{-1}\varphi^{\prime}|^{2}-|\psi
^{-1}\psi^{\prime}|^{2}\right)  ds\text{ for all }t
\]
and therefore $|\varphi^{-1}\varphi^{\prime}|^{2}=|\psi^{-1}\psi^{\prime}%
|^{2}$ a.e. In particular, we then have
\[
0=\mathbb{E}\left[  \left(  \int_{0}^{t}\langle\psi^{-1}\psi^{\prime
}(s)-\varphi^{-1}\varphi^{\prime}(s),dB_{s}^{L}\rangle\right)  ^{2}\right]
=\int_{0}^{t}\left\vert \psi^{-1}\psi^{\prime}(s)-\varphi^{-1}\varphi^{\prime
}(s)\right\vert ^{2}ds
\]
from which it follows $\psi^{-1}\psi^{\prime}(t)-\varphi^{-1}\varphi^{\prime
}(t)=0$ for any $t\in\lbrack0,T]$. Finally, we see that for any $t\in
\lbrack0,T]$%
\[
\left(  \varphi\psi^{-1}\right)  ^{\prime}\left(  t\right)  =\varphi^{\prime
}\psi^{-1}\left(  t\right)  -\varphi\psi^{-1}\psi^{\prime}\psi^{-1}\left(
t\right)  =\varphi^{\prime}\psi^{-1}\left(  t\right)  -\varphi\varphi
^{-1}\varphi^{\prime}\psi^{-1}\left(  t\right)  =0
\]
and therefore $\varphi^{-1}\psi\equiv e$.
\end{proof}

\begin{prop}
\label{p.6.7} For any $\varphi, \psi, \varphi_{1},..., \varphi_{n}, \psi
_{1},.., \psi_{n} \in H(G)$, $f \in L^{2}\left(  W(G), \mu\right)  $%

\begin{align*}
&  \left(  U_{\varphi_{1}}^{R}...U_{\varphi_{n}}^{R}\right)  f\left(
g\right)  =\left(  Z_{\varphi_{n}...\varphi_{1}}^{R}\right) ^{1/2}\left(
g\right)  f\left(  R_{\varphi_{1}...\varphi_{n}}g\right)  ,\\
&  \left(  U_{\psi_{1}}^{L}...U_{\psi_{n}}^{L}\right)  f\left(  g\right)
=\left(  Z_{\psi_{n}...\psi_{1}}^{L}\right)  ^{1/2}\left(  g\right)  f\left(
L_{\psi_{1}...\psi_{n}}g\right) ,\\
&  \left(  U_{\varphi}^{R}\right)  ^{-1}=\left(  U_{\varphi}^{R}\right)
^{\ast}=U_{\varphi^{-1}}^{R},\\
&  \left(  U_{\psi}^{L}\right)  ^{-1}=\left(  U_{\psi}^{L}\right)  ^{\ast
}=U_{\psi^{-1}}^{L}.
\end{align*}
In particular, this implies that $U_{\varphi}^{R},U_{\psi}^{L}$ are unitary
operators on $L^{2}\left(  W(G),\mu\right)  $.
\end{prop}

\begin{proof}
For any $f,h\in L^{2}\left(  W(G),\mu\right)  $, $\varphi,\varphi_{1}%
,\varphi_{2}\in H\left(  G\right)  $ we have%
\begin{align*}
&  \left(  U_{\varphi_{1}}^{R}U_{\varphi_{2}}^{R}f\right)  \left(  g\right)
=\left(  Z_{\varphi_{1}}^{R}\left(  g\right)  Z_{\varphi_{2}}^{R}\left(
g\varphi_{1}\right)  \right)  ^{1/2}f\left(  g\varphi_{1}\varphi_{2}\right)
=\\
&  \left(  Z_{\varphi_{2}\varphi_{1}}^{R}\right)  ^{1/2}\left(  g\right)
f\left(  g\varphi_{1}\varphi_{2}\right)
\end{align*}
by the properties of the Radon-Nikodym densities, and
\begin{align*}
&  \langle\left(  U_{\varphi}^{R}\right)  ^{\ast}f,h\rangle_{L^{2}\left(
W(G),\mu\right)  }=\langle f,U_{\varphi}^{R}h\rangle_{L^{2}\left(
W(G),\mu\right)  }=\\
&  \int_{W(G)}f(g)h(g\varphi)h_{\varphi}(g)d\mu(g)=\\
&  \int_{W(G)}f(g\varphi^{-1})h(g)\left(  Z_{\varphi}^{R}\right)
^{1/2}(g\varphi^{-1})Z_{\varphi^{-1}}^{R}\left(  g\right)  d\mu(g)=\\
&  \int_{W(G)}f(g\varphi^{-1})\left(  Z_{\varphi^{-1}}^{R}\right)
^{1/2}(g)h(g)d\mu(g)=\langle U_{\varphi^{-1}}^{R}f,h\rangle_{L^{2}\left(
W(G),\mu\right)  }.
\end{align*}
The case of $U^{L}$ can be checked similarly.
\end{proof}

\subsection{Properties of the Brownian representations}

\begin{nota}
\label{n.4.2} We denote by%
\begin{align*}
\mathcal{M}^{R}  &  : =\left(  U_{\varphi}^{R},\varphi\in H(G)\right)
^{\prime\prime}\\
\mathcal{M}^{L}  &  : =\left(  U_{\varphi}^{L},\varphi\in H(G)\right)
^{\prime\prime}%
\end{align*}
the von Neumann algebras generated by the operators $U_{\varphi}^{R}$,
$U_{\varphi}^{L}$ respectively.
\end{nota}

Theorem \ref{t.7.1} collects some basic facts about the left and right
Brownian representations. Most of these properties are what one expects from
the classical case of regular representations of locally compact groups. But
some of the proofs are fundamentally different. For example, the fact that the
von Neumann algebras generated by the left and right representations are
commutants of each other has been originally proved by I. Segal in \cite{Segal1950a} for the regular representation of a unimodular locally compact Lie group with a bi-invariant Haar measure. One of the major facts he
used was existence of an approximating identity and the one-to-one correspondence between unitary representation of the group $G$ and the non-degenerate $\ast$-representations of the group algebra $L^{1}\left(  G \right)$ (e.g. \cite[Section 3.2]{FollandHABook}). These fundamental constructions are not available in our case. Theorem \ref{t.7.1} does not answer the question whether $\mathcal{M}^{L}$ and $\mathcal{M}^{R}$ are commutants of each other, which will be addressed in another article.

\begin{thm}
\label{t.7.1}

\begin{enumerate}
\item the unitary operators $U^{R}_{\varphi}$ and $U^{L}_{\psi}$ commute for
any $\varphi, \psi\in H\left(  G \right)  $, and so $\left(  \mathcal{M}%
^{R}\right)  ^{\prime}\subseteq\mathcal{M}^{L}$ and $\left(  \mathcal{M}%
^{L}\right)  ^{\prime}\subseteq\mathcal{M}^{R}$. The representations $U^{L}$
and $U^{R}$ are unitarily equivalent, and the intertwining operator is the
unitary involution $J$ defined by \eqref{e.5.1};

\item $\Omega=\mathbbm{1}$ is a separating cyclic vector of norm $1$ for both
$\mathcal{M}^{R}$ and $\mathcal{M}^{L}$ in $L^{2}\left(  W(G), \mu\right)  $.
If $G$ is abelian, then the corresponding von Neumann algebra $\mathcal{M}
^{R}=\mathcal{M}^{L}$ is maximal abelian in $B\left(  L^{2}\left(  W(G);
\mu\right)  \right)  $.

\item For any $T\in\mathcal{M}^{R}$ the map $T\longmapsto T\mathbbm{1}$ is injective.

\item The vacuum vector $\Omega=\mathbbm{1}$ defines a faithful normal weight
$\tau$ on $\mathcal{M}^{R}$ (and similarly on $\mathcal{M}^{L}$) by%

\begin{equation}
\tau\left(  m\right)  :=\langle m\Omega,\Omega\rangle_{L^{2}\left(
W(G),\mu\right)  }=\int_{W(G)}m\left(  \mathbbm{1}\right)  \left(  g\right)
d\mu\left(  g\right)  \label{e.4.1}%
\end{equation}
for any $m\in\mathcal{M}^{R}$. In addition, $\tau(\operatorname{I})$ is
finite, and so $\tau$ is a faithful normal state.
\end{enumerate}
\end{thm}

\begin{proof}
1. First we observe that $U^{L}_{\varphi}$ and $U^{R}_{\psi}$ commute. Indeed,
for any $\varphi, \psi\in H\left(  G \right)  $, $f \in L^{2}\left(  W\left(
G \right)  , \mu\right)  $ we have%

\begin{align*}
&  \left(  U_{\psi}^{L}U_{\varphi}^{R}f\right)  \left(  g\right) \\
&  =\left(  \frac{d\mu\left(  \psi^{-1}g\right)  }{d\mu\left(  g\right)
}\right)  ^{1/2}\left(  \frac{d\mu\left(  \psi^{-1}g\varphi\right)  }%
{d\mu\left(  \psi^{-1}g\right)  }\right)  ^{1/2}f\left(  \psi^{-1}%
g\varphi\right) \\
&  =\left(  \frac{d\mu\left(  \psi^{-1}g\varphi\right)  }{d\mu\left(
g\right)  }\right)  ^{1/2}f\left(  \psi^{-1}g\varphi\right)  =\left(
U_{\varphi}^{R}U_{\psi}^{L}f\right)  \left(  g\right)  .
\end{align*}
To see that $U^{L}$ and $U^{R}$ are unitarily equivalent we use Proposition
\ref{p.3.2}, and the following simple observation. Using Notation \ref{n.4.3}
for the left and right multiplication operators on $W\left(  G\right)  $, we
see that%

\[
J R_{\varphi\ \ast} = L_{\varphi\ \ast} J.
\]
Then by Proposition \ref{p.3.2} for any $f \in L^{2}\left(  W\left(  G
\right)  , \mu\right)  $%

\begin{align*}
&  \left(  JU^{R}_{\varphi} f \right)  \left(  g \right)  = J\left(
Z^{R}_{\varphi}\left(  g \right)  \left(  R_{\varphi\ \ast}f\right)  \left(  g
\right)  \right)  =\\
&  Z^{L}_{\varphi}\left(  g \right)  J\left(  R_{\varphi\ \ast}f\left(  g
\right)  \right)  =Z^{L}_{\varphi}\left(  g \right)  \left(  L_{\varphi\ \ast
}J f\left(  g \right)  \right)  =\left(  U^{L}_{\varphi}Jf\right)  \left(  g
\right)  .
\end{align*}

2. Theorem \ref{t.4.11} shows that $\mathbbm{1}$ is cyclic for $\mathcal{M}%
^{R}$, and similarly one can show that it is cyclic for $\mathcal{M}^{L}$.

Now suppose that $G$ is abelian. It is clear that in this case $\mathcal{M}%
=\mathcal{M}^{R}=\mathcal{M}^{L}$ is abelian, and therefore $\mathcal{M}%
^{\prime}=\mathcal{M}$ which implies that it is maximal abelian. Note that
another explanation for $\mathcal{M}$ being maximal abelian is that as we know
it has a cyclic vector. Then by \cite[Corollary 7.2.16]{KadisonRingrose2}
$\mathcal{M}$ is maximal abelian as an abelian subalgebra with a cyclic vector.

3. This is a standard fact from the Tomita-Takesaki theory, but in this case
it is easy to verify and we include the argument for completeness. Let $T
\in\mathcal{M}^{R}$ be such that $T \mathbbm{1}=0$. Then $T$ commutes with all
operators in $\mathcal{M}^{L}$, and therefore%

\[
U_{\psi^{-1}}^{L}TU_{\psi}^{L}\mathbbm{1}=T \mathbbm{1}=0,
\]
and so%

\[
TU_{\psi}^{L}\mathbbm{1}=0
\]
for all $\psi\in H\left(  G \right) $. Since $\mathbbm{1}$ is cyclic for both
left and right representations, we see that $T=0$.

4. The first part of this statement is a standard fact following from the GNS
construction (e.g. \cite{SunderBook1987}). To see that $\tau$ is a state, we
note that the identity operator $\operatorname{I}$ in $\mathcal{M}^{R}$ can be
represented as $U^{R}_{e}$, where $e(t)\equiv e$ for $t \in[0, T]$. Thus%

\[
\tau\left(  \operatorname{I} \right)  =\tau\left(  U^{R}_{e}\right)  =1.
\]
The same holds for $\mathcal{M}^{L}$.
\end{proof}

\begin{prop}
[$\tau$ is not a trace]\label{p2} For any $\varphi,\psi\in H^{T}\left(
G\right),$%
\[
\tau\left(  U_{\varphi}^{R}U_{\psi}^{R}\right)  =\tau\left(  U_{\psi}%
^{R}U_{\varphi}^{R}\right)
\]
if and only if%
\begin{equation}
\int_{0}^{T}\langle\varphi^{-1}\varphi^{\prime},\psi^{\prime}\psi^{-1}\rangle
ds=\int_{0}^{T}\langle\varphi^{\prime}\varphi^{-1},\psi^{-1}\psi^{\prime
}\rangle ds. \label{e4}%
\end{equation}

\end{prop}

\begin{proof}
By definition of $\tau$ and Propositions \ref{p.5.3} and \ref{p.6.7} we see that%

\begin{align*}
\tau\left(  U_{\varphi}^{R}U_{\psi}^{R}\right)   &  =E_{\mu}Z_{\psi\varphi
}^{R}\left(  g\right)  =\exp\frac{-\Vert\psi\varphi\Vert_{H,T}^{2}}{8}\\
&  =\exp\frac{-\Vert\varphi\Vert_{H,T}^{2}-\Vert\psi\Vert_{H,T}^{2}}{8}%
\exp\frac{1}{4}\int_{0}^{T}\langle\operatorname{Ad}_{\varphi}\varphi^{\prime
}\varphi^{-1},\psi^{\prime}\psi^{-1}\rangle dt\\
&  =\exp\frac{-\Vert\varphi\Vert_{H,T}^{2}-\Vert\psi\Vert_{H,T}^{2}}{8}%
\exp\frac{1}{4}\int_{0}^{T}\langle\varphi^{\prime}\varphi,\psi^{-1}%
\psi^{\prime}\rangle dt.
\end{align*}
Applying this computation to $\tau\left(  U_{\psi}^{R}U_{\varphi}^{R}\right)
$ completes the proof.
\end{proof}

\section{Energy representation\label{s.energy}}

Let $\left(  H, W, \Gamma\right)  $ be an abstract Wiener space, that is, $H$
is a real separable Hilbert space densely continuously embedded into a real
separable Banach space $W$, and $\Gamma$ is the Gaussian measure defined by
the characteristic functional%

\[
\int_{W}e^{i\varphi\left(  x\right)  }d\Gamma\left(  x\right)  =\exp\left(
-\frac{|\varphi|_{H^{\ast}}^{2}}{2}\right)
\]
for any $\varphi\in W^{\ast}\subset H^{\ast}$. We will identify $W^{\ast}$
with a dense subspace of $H$ such that for any $h\in W^{\ast}$ the linear
functional $\langle\cdot,h\rangle$ extends continuously from $H$ to $W$. We
will usually write $\langle\varphi,w\rangle:=\varphi\left(  w\right)  $ for
$\varphi\in W^{\ast}$, $w\in W$. More details can be found in
\cite{BogachevBook}.

It is known that $\Gamma$ is a Borel measure, that is, it is defined on the
Borel $\sigma$-algebra $\mathcal{B}\left(  W\right)  $ generated by the open
subsets of W. The Gaussian measure $\Gamma$ is quasi-invariant under the
translations from $H$ and invariant under orthogonal transformations of $H$.
We want to be more precise here.

\begin{nota}
We call an orthogonal transformation of $H$ which is a topological
homeomorphism of $W^{\ast}$ a \textbf{rotation} of $W^{\ast}$. The space of
all such rotations is denoted by $\operatorname{O}\left(  W^{\ast}\right)  $.
For any $R \in\operatorname{O}\left(  W^{\ast}\right)  $ its adjoint,
$R^{\ast}$, is defined by%

\[
\langle\varphi, R^{\ast}w \rangle:=\langle R^{-1}\varphi, w \rangle, \ w \in
W, \varphi\in W^{\ast}.
\]

\end{nota}

\begin{thm}
For any $R \in\operatorname{O}\left(  W^{\ast}\right)  $ the map $R^{\ast}$ is
a $\mathcal{B}\left(  W \right)  $-measurable map from $W$ to $W$ and%

\[
\Gamma\circ\left(  R^{\ast}\right) ^{-1}=\Gamma.
\]

\end{thm}

\begin{proof}
The measurability of $R^{\ast}$ follows from the fact that $R$ is continuous
on $H$. For any $\varphi\in W^{\ast}$%

\begin{align*}
\int_{W}e^{i\varphi\left(  x\right)  }d\Gamma\left(  \left(  R^{\ast}\right)
^{-1}x\right)   &  =\int_{W}e^{i\langle\varphi,x\rangle}d\Gamma\left(  \left(
R^{\ast}\right)  ^{-1}x\right)  =\int_{W}e^{i\langle\varphi,R^{\ast}x\rangle
}d\Gamma\left(  x\right)  \\
&  =\exp\left(  -\frac{|R^{-1}\varphi|_{H^{\ast}}^{2}}{2}\right)  =\exp\left(
-\frac{|\varphi|_{H^{\ast}}^{2}}{2}\right)  \\
&  =\int_{W}e^{i\varphi\left(  x\right)  }d\Gamma\left(  x\right)
\end{align*}
since $R$ is an isometry.
\end{proof}

\begin{cor}
\label{c.6.3} Any $R\in\operatorname{O}\left(  W^{\ast}\right) $ extends to a
unitary map on $L^{2}\left(  W,\Gamma\right)  $.
\end{cor}

The Cameron-Martin theorem states that $\Gamma$ is quasi-invariant under
translations by elements in $H$, namely, $T_{h}:W\rightarrow W$, $T_{h}\left(
w\right)  =w+h$. The Radon-Nikodym derivative is given by%

\[
\frac{d\left(  T_{h}\right) _{\ast}\Gamma}{d\Gamma}\left(  w\right)
=\frac{d\left(  \Gamma\circ T_{h}^{-1}\right)  }{d\Gamma}\left(  w\right)
=\frac{d\left(  \Gamma\circ T_{-h}\right)  }{d\Gamma}\left(  w\right)
=e^{-\langle h,w\rangle-\frac{|h|^{2}}{{2}}},\ w\in W,h\in H.
\]
Following \cite{DriverHall1999a} we consider the Gaussian regular
representation of the Euclidean group of transformations $w\mapsto R^{\ast
}w+h$, $x\in H,h\in H,R\in\operatorname{O}\left(  W^{\ast}\right)  $ on
$L^{2}\left(  W,\Gamma\right)  $ defined as%

\begin{align}
&  \left(  U_{R,h}f\right)  \left(  w\right) : =\left(  \frac{d\left(
\Gamma\circ\left(  T_{h}R^{\ast}\right)  \right)  }{d\Gamma}\left(  w\right)
\right) ^{1/2}f\left(  \left(  T_{h}R^{\ast}\right) ^{-1}\left(  w\right)
\right)  =\nonumber\\
&  \left(  \frac{d\left(  \Gamma\circ T_{h}\right)  }{d\Gamma}\left(
w\right)  \right) ^{1/2}f\left(  \left(  R^{\ast}\right) ^{-1}\left(
w-h\right)  \right)  =\label{e.C.1}\\
&  e^{\langle h,w\rangle-\frac{|h|^{2}}{{2}}}f\left(  \left(  R^{\ast}\right)
^{-1}\left(  w-h\right)  \right) , \ w\in W\nonumber
\end{align}
which is well-defined by Corollary \ref{c.6.3}. It is clear that this is a
unitary representation.

Now we need to define the Fourier-Wiener transform $\mathcal{F}$ on
$L^{2}\left(  W,\Gamma\right)  $. This can be done in several ways, and for
now we refer to Definition 17 in \cite{DriverHall1999a} with the parameter
$r=1/2$. In particular, one can check that $\mathcal{F}^{4}\equiv I$ on
$L^{2}\left(  W,\Gamma\right)  $ by doing a computation on Hermite functions.

The following formula is very convenient for computations, but some care
should be taken over its applicability. One of the ways of making this formula
rigorous is to define it on Hermite functions using the Fock space, as it is
done in \cite{GuichardetBook}.%

\[
\left(  \mathcal{F}f\right)  \left(  w\right)  =\int_{W}f\left(  iw+\sqrt
{2}u\right)  d\Gamma\left(  u\right) , \ f\in L^{2}\left(  W,\Gamma\right) .
\]
In particular, identities in Proposition \ref{p.6.4} follow from this formula
quite easily.

\begin{prop}
\label{p.6.4}

1. Let $\mathcal{E}:=\operatorname{Span}_{\mathbb{C}}\{\widehat{\varphi
}\left(  w\right) : =e^{i\langle\varphi,w\rangle},\varphi\in W^{\ast},w\in
W\}$. Then $\mathcal{E}$ is an algebra which is dense in $L^{2}\left(
W,\Gamma\right)  $.

2.For any $\varphi\in W^{\ast}$ we have%
\begin{align}
&  \int_{W}\widehat{\varphi}\left(  w\right)  d\Gamma\left(  w\right)
=e^{-\frac{|\varphi|_{H^{\ast}}^{2}}{2}},\nonumber\\
&  \left(  \mathcal{F}\widehat{\varphi}\right)  \left(  w\right)
=e^{-|\varphi|_{H^{\ast}}^{2}}e^{-\langle\varphi,w\rangle},\text{
and}\label{e.6.1}\\
&  \left(  \mathcal{F}e^{\langle\varphi,\cdot\rangle}\right)  \left(
w\right)  =e^{|\varphi|_{H^{\ast}}^{2}}\widehat{\varphi}\left(  w\right)
.\nonumber
\end{align}

\end{prop}

\begin{proof}
The first statement is proven in a number of references, one of which is
\cite{HidaBook}, Theorem 4.1, so we omit the proof for now. Identities in
\eqref{e.6.1} follow from similar finite-dimensional calculations using the
methods in \cite{DriverHall1999a} or approximations by Hermite functions.
\end{proof}

\begin{prop}
[Proposition 18 \cite{DriverHall1999a}]\label{p.6.5}If $f\in L^{2}\left(
W,\Gamma\right)  ,$ $R\in\operatorname{O}\left(  W^{\ast}\right)  ,$ $h\in
W^{\ast},$ then%
\[
\left(  \mathcal{F}U_{R,h}\mathcal{F}^{-1}f\right)  \left(  w\right)
=e^{-\frac{i\langle h,w\rangle}{2}}f\left(  R^{\ast}w\right)  \text{ for }w\in
W.
\]

\end{prop}

\begin{proof}
By Proposition \ref{p.6.4} it is enough to check the statement for $f\left(
w\right)  =\widehat{\varphi}\left(  w \right)  $. First, let us compute
$\mathcal{F}^{3}\widehat{\varphi}\left(  w \right)  $ using \eqref{e.6.1}%

\begin{align*}
\left(  \mathcal{F}^{3}\widehat{\varphi}\right)  \left(  w\right)   &
=e^{-|\varphi|_{H^{\ast}}^{2}}\left(  \mathcal{F}^{2}e^{-\langle\varphi
,\cdot\rangle}\right)  \left(  w\right)  \\
&  =e^{-|\varphi|_{H^{\ast}}^{2}}e^{|\varphi|_{H^{\ast}}^{2}}\left(
\mathcal{F}e^{-i\langle\varphi,\cdot\rangle}\right)  \left(  w\right)
=e^{-|\varphi|_{H^{\ast}}^{2}}e^{\langle\varphi,w\rangle}.
\end{align*}

Then
\begin{align*}
\left(  \mathcal{F}U_{R,h}\mathcal{F}^{-1}\widehat{\varphi}\right)  \left(
w\right)   &  =\left(  \mathcal{F}U_{I,h}U_{R,0}\mathcal{F}^{3}\widehat
{\varphi}\right)  \left(  w\right)  \\
&  =e^{-|\varphi|_{H^{\ast}}^{2}}\left(  \mathcal{F}U_{I,h}U_{R,0}%
e^{\langle\varphi,\cdot\rangle}\right)  \left(  w\right)  \\
&  =e^{-|\varphi|_{H^{\ast}}^{2}}e^{-\frac{|h|^{2}}{{4}}}\left(
\mathcal{F}e^{\frac{\langle h,\cdot\rangle}{2}}e^{\langle R\varphi
,\cdot+h\rangle}\right)  \left(  w\right)  \\
&  =e^{-|\varphi|_{H^{\ast}}^{2}}e^{-\frac{|h|^{2}}{{4}}}e^{\langle
R\varphi,h\rangle}\left(  \mathcal{F}e^{i\frac{\langle-i\left(  h+2R\varphi
\right)  ,\cdot\rangle}{2}}\right)  \left(  w\right)  \\
&  =e^{-\frac{|h+2R\varphi|_{H^{\ast}}^{2}}{4}}e^{\frac{|h+2R\varphi
|_{H^{\ast}}^{2}}{4}}e^{i\langle\frac{h}{2}+R\varphi,w\rangle}=e^{i\langle
\frac{h}{2},w\rangle}\widehat{\varphi}\left(  R^{\ast}w\right)  ,
\end{align*}
where we used the fact that $|R\varphi|_{H^{\ast}}=|\varphi|_{H^{\ast}}$.
\end{proof}

\begin{cor}
\label{c.A.6} By taking $f\equiv1$ in Proposition \ref{p.6.5}, we see that for
any $h \in H$%

\[
\mathcal{F}e^{\langle h,w\rangle-\frac{|h|^{2}}{{2}}}=e^{-\frac{i\langle h, w
\rangle}{2}}.
\]

\end{cor}

We now work on the measure space $\left(  W\left(  \mathfrak{g}\right) ,
\mathcal{B}_{W\left(  \mathfrak{g}\right)  }, \nu\right) $ and let
$w_{s}:W\left(  \mathfrak{g}\right)  \rightarrow\mathfrak{g}$ be the
projection map, $w_{s}\left(  \omega\right)  =\omega_{s}$ for all $0\leqslant
s\leqslant T$ and $\omega\in W\left(  \mathfrak{g}\right)  . $ [Note, we may
also view $w$ as the identity map from $W\left(  \mathfrak{g}\right)  $ to
$W\left(  \mathfrak{g}\right)  . ]$ The energy representation is a unitary
representation of $H\left(  G\right)  $ on the space $L^{2}\left(  W\left(
\mathfrak{g}\right)  , \nu\right)  $. First we introduce an operator on
$W\left(  \mathfrak{g}\right)  $ used to define the energy representation.
Note that since the inner product on $\mathfrak{g}$ is $\operatorname{Ad}%
$-invariant, the operator $O_{\varphi}$ defined by%

\begin{equation}
O_{\varphi}\left(  w\right)  : =\int_{0}^{\cdot}\operatorname{Ad}_{\varphi
}\delta w_{s},w\in W\left(  \mathfrak{g}\right)  , \varphi\in H\left(
G\right)  \label{e.7.1}%
\end{equation}
is well-defined on $W\left(  \mathfrak{g}\right) $ by L\'{e}vy's criterion as
we indicated in \eqref{e.2.15}. Moreover, since the It\^{o} and Stratonovich
integrals of deterministic integrands are equal, we see that%

\[
O_{\varphi}\left(  w\right)  =\int_{0}^{\cdot}\operatorname{Ad}_{\varphi
}\delta w_{s}=\int_{0}^{\cdot}\operatorname{Ad}_{\varphi}dw_{s}.
\]

\begin{df}
\label{d.6.1} For any $\varphi\in H\left(  G \right)  $%

\[
\left(  E_{\varphi}f\right)  \left(  w\right)  : =e^{i\int_{0}^{T}%
\langle\varphi^{-1}\varphi^{\prime}\left(  s\right) , dw_{s}\rangle}f\left(
O_{\varphi^{-1}}w\right)  .
\]
for any $f\in L^{2}\left(  W\left(  \mathfrak{g}\right)  , \nu\right)  $. Then
$E_{\varphi}$ is called the \emph{energy representation} of $H\left(  G\right)
$.
\end{df}

Again using the fact that the It\^{o} and Stratonovich integrals are equal for
deterministic integrands, we see that
\[
\left(  E_{\varphi}f\right)  \left(  w\right)  =e^{i\int_{0}^{T}\langle
\varphi^{-1}\varphi^{\prime}\left(  s\right) , dw_{s}\rangle}f\left(
O_{\varphi^{-1}}w\right) .
\]
It is easy to see that $E_{\varphi}^{\ast}=E_{\varphi^{-1}}$, so it is a
unitary representation of $H\left(  G\right)  $ on $L^{2}\left(  W\left(
\mathfrak{g}\right) , \nu\right) $. For our future results using It\^{o}
integrals will be more convenient, so this is what we will be using from now
on mostly.

\begin{thm}
\label{t.uniequiv}Both $U^{R}$ and $U^{L}$ are unitarily equivalent to the
energy representation $E$.
\end{thm}

\begin{proof}
As we noted in Theorem \ref{t.7.1}, $U^{R}$ and $U^{L}$ are unitarily
equivalent. Using \eqref{e.4.2} we see that under the inverse It\^{o} map
$B^{L}$ the left multiplication is mapped to the following operator%

\begin{equation}
\left(  \left( B^{L}\right)^{\ast} R_{\varphi}^{\ast}\right)  f\left(  w\right)
=f\left(  O_{\varphi^{-1}}w+\int_{0}^{\cdot}\varphi^{-1}d\varphi\right)  ,
\label{e.8.12}%
\end{equation}
where $f\in L^{2}\left(  W\left(  \mathfrak{g}\right), \nu\right)$, $w\in
W\left(  \mathfrak{g}\right)$, and $R_{\varphi}^{\ast}$ is the adjoint operator.

Then the representation $U^{R}_{\varphi}$ corresponds to the following
representation on $L^{2}\left(  W\left(  \mathfrak{g}\right)  , \nu\right)  $%

\begin{align}
\left(  u_{\varphi}^{R}f\right)  \left(  w\right)   &  : =\left(
\left( B^{L}\right)^{\ast} U_{\varphi}^{R}f\right)  \left(  w\right)
\label{e.8.13}\\
&  =e^{\frac{1}{2}\int_{0}^{T}\langle\varphi^{-1}\varphi^{\prime}\left(
s\right), dw_{s}\rangle-\frac{1}{4}\Vert\varphi\Vert_{H,\cdot}^{2}}f\left(
O_{\varphi^{-1}}w+\int_{0}^{\cdot}\varphi^{-1}d\varphi\right)  .\nonumber
\end{align}
Here we used $O_{\varphi^{-1}}$ to denote the operator introduced by
\eqref{e.7.1}. Note that $\left(  u_{\varphi}^{R}f\right)  \left(  w\right)
=U_{R,h}$, where $U_{R,h}$ is defined by \eqref{e.C.1} with $R^{\ast}\left(
w\right)  =O_{\varphi^{-1}}w$ and $h=-\varphi^{-1}d\varphi$. The adjoint
representation of $G$ on $\mathfrak{g}$ is unitary, and therefore
$O_{\varphi^{-1}}$ is a continuous unitary transformation on $H\left(
\mathfrak{g}\right) $. Thus we can apply Proposition \ref{p.6.5} to see that
$u_{\varphi}^{R}$ is unitarily equivalent to $E_{\varphi}$. The intertwining
operator here is the Fourier-Wiener transform $\mathcal{F}$, and the
intertwining map between $U^{L}$ and $E$ is then $\mathcal{F}\circ \left( B^{L}\right)^{\ast}$.
\end{proof}

\begin{cor}
\label{c.6.4} Theorem \ref{t.4.11} implies that $1$ is a cyclic vector for the
energy representation.
\end{cor}

\bibliographystyle{amsplain}
\def\cprime{$'$}
\providecommand{\bysame}{\leavevmode\hbox to3em{\hrulefill}\thinspace}
\providecommand{\MR}{\relax\ifhmode\unskip\space\fi MR }
\providecommand{\MRhref}[2]{%
  \href{http://www.ams.org/mathscinet-getitem?mr=#1}{#2}
}
\providecommand{\href}[2]{#2}

\end{document}